\documentclass{article}
\usepackage[utf8]{inputenc}
\usepackage{graphics} 
\usepackage{epsfig} 
\usepackage{amsmath} 
\usepackage{amssymb}  

\usepackage{xcolor}
\usepackage{mathrsfs}
\usepackage{amsthm}
\newcommand{\skp}{\vspace{\baselineskip}}

\newtheorem{definition}{Definition}
\newtheorem{remark}{Remark}
\newtheorem{theorem}{Theorem}
\newtheorem{assumption}{Assumption}
\newtheorem{lemma}{Lemma}
\newtheorem{proposition}{Proposition}

\title{Behavior near walls in the mean-field approach to crowd dynamics\thanks{Financial support from the Swedish Research Council (2016-04086) is gratefully acknowledged. We thank the anonymous reviewers for comments and suggestions that greatly helped to improve the presentation of the results.}}
\author{Alexander Aurell \thanks{A. Aurell is with the Department of Mathematics,
        KTH Royal Institute of Technology, SE-100 44 Stockholm, Sweden {\tt\small aaurell@kth.se}} \and Boualem Djehiche \thanks{B. Djehiche is with the Department of Mathematics, KTH Royal Institute of Technology, SE-100 44 Stockholm, Sweden {\tt\small boualem@kth.se}}}

\begin{document}

\maketitle
\vspace{-0.5cm}
\begin{abstract}
This paper introduces a system of stochastic differential equations (SDE) of mean-field type that models pedestrian motion. The system lets the pedestrians spend time at, and move along, walls, by means of sticky boundaries and boundary diffusion. As an alternative to Neumann-type boundary conditions, sticky boundaries and boundary diffusion have a 'smoothing' effect on pedestrian motion. When these effects are active, the pedestrian paths are semimartingales with first-variation part absolutely continuous with respect to the Lebesgue measure $dt$, rather than an increasing processes (which in general induces a measure singular with respect to $dt$) as is the case under Neumann boundary conditions. We show that the proposed mean-field model for pedestrian motion admits a unique weak solution and that it is possible to control the system in the weak sense, using a Pontryagin-type maximum principle.  We also relate the mean-field type control problem to the social cost minimization in an interacting particle system. We study the novel model features numerically and we confirm empirical findings on pedestrian crowd motion in congested corridors.

\skp


\skp


\end{abstract}

\section{Introduction}
\label{sec:introduction}

Models for pedestrian motion in confined domains must consider interaction with 
solid obstacles such as pillars and walls. The pedestrian response to a 
restriction of movement has been included into crowd models either as boundary 
conditions or repulsive forces. Up until today, the Neumann condition and its 
variants (e.g. no-flux) have been especially popular among the boundary conditions. 
The Neumann condition suffers from a drawback related to its microscopic 
(pathwise) interpretation. A Neumann condition on the crowd density corresponds 
to pedestrian paths reflecting in the boundary. In reality, pedestrians do not 
bounce off walls in the manner of classical Newtonian particles, but their 
movement is slowed down by the impact and a positive amount of time is needed to 
choose a new direction of motion. It is natural to think that whenever a 
pedestrian is forced (or decides) to make contact with a wall, she stays there 
for some time. During this time, she can move and interact with other 
pedestrians, before re-entering the interior of the domain.

\subsection{Mathematical modeling of pedestrian-wall interaction}

Today there is more than one conventional approach to the mathematical modeling 
of pedestrian motion. This section aims to summarize how they incorporate the 
interaction between pedestrians and walls.

Microscopic force-based models, among which the social force model has gained 
the most attention, describes pedestrians as Newton-like particles. From the 
initial work \cite{helbing1995social} and onward, the influence a wall has on 
the pedestrian is modelled as a repulsive force. The shape of the corresponding 
potential has been studied experimentally, for example in 
\cite{ma2010experimental}. The cellular automata is another widely used 
microscopic approach to pedestrian crowd modeling. Walls are modeled as cells to 
which pedestrians cannot transition, already the original work 
\cite{kirchner2002simulation} considers this viewpoint. In the continuum limits 
of cellular automata, as for example in  \cite{burger2011continuous, 
burger2016lane}, boundary conditions are often set to no-flux conditions of the 
same type as \eqref{eq:burger} below.

The focus of macroscopic models is the global pedestrian density, either in a 
stationary or a dynamic regime. Inspired by fluid dynamics \cite{hughes2000flow} 
treats the crowd as a 'thinking fluid' that moves at maximum speed towards a 
target location while taking environmental factors into account, such as the 
congestion of the crowd. In this category of models, boundary conditions at 
impenetrable walls are most often implemented as Neumann conditions for the 
pedestrian density. The pathwise interpretation of a Neumann boundary condition 
is instantaneous reflection. A nonlocal projection of pedestrian velocity in 
normal and tangential direction of the boundary respectively is suggested in 
\cite{bellomo2015toward} and implemented in \cite{bellomo2016behavioral}, 
allowing for nonlocal interaction with boundaries.

Mean-field games and mean-field type control/games are macroscopic models of 
rational pedestrians with the ability to anticipate crowd movement, and adapt 
accordingly. These models can capture competition between individuals as well as 
crowd/sub-crowd cooperation.
In the mean-field approach to pedestrian crowd modeling pedestrian-to-pedestrian 
interaction is assumed to be symmetric and weak, thus plausibly replaced by an 
interaction with a mean field (typically a functional of the pedestrian 
density). One of the most attractive features of the mean-field approach is that 
it connects the macroscopic (pedestrian density) and the microscopic (pedestrian 
path) point-of-view, typically through results on the near-optimality/equilibrium
of mean-field optimal controls/equilibria. The connection permits us to infer individual 
pedestrian behavior from crowd density simulations, and vice versa. In what follows, the crowd 
density is denoted by $m$. In \cite{lachapelle2011mean}, the density is 
subjected to $n(x)\cdot \nabla m(t,x) = 0$ at walls, where $n(x)$ is the outward 
normal at $x$. Under this constraint, the normal velocity of the pedestrian is zero 
at any wall. Taking conservation of probability mass into account, \cite{burger2013mean} 
derives the following boundary condition
\begin{equation}
\label{eq:burger}
  -n(x)\cdot \left(\nabla m(t,x) - G(m)v(t,x)\right) = 0,
\end{equation}
where $G(m)v$ is a general form of the pedestrian velocity. The constraint 
\eqref{eq:burger} represents reflection at the boundary since in the 
corresponding microscopic interpretation pedestrians make a classical Newtonian 
bounce whenever they hit the boundary.  The same type of constraint is used in  
\cite{albi2017mean}. The case of several interacting populations in a bounded 
domain with reflecting boundaries has been studied in the stationary and dynamic 
case \cite{cirant2015multi, achdou2017mean, bardi2017uniqueness}. In these 
papers, the crowd density at walls is constrained by 
\begin{equation}
    n(x)\cdot\left(\nabla m(t,x) + m(t,x)\partial_p H(x, \nabla u)\right) = 0.
\end{equation} 
The constraint is a 
reflection and the term $-\partial_p H(x,\nabla u)$ is the velocity of pedestrians that use the mean-field 
equilibrium strategy.

\subsection{Sticky reflected stochastic differential equations}

The sticky reflected Brownian motion was
discovered by Feller \cite{feller1952parabolic, feller1954diffusion, 
feller1957generalized}. He studied the infinitesimal generator of
strong Markov processes on $[0,\infty)$ that behave like Brownian motion
in $(0,\infty)$, and showed that it is possible for
the process to be 'sticky' on the boundary, i.e. to sojourn at $0$.
So 'sticky reflection' was appended
to the list of boundary conditions for diffusions, which already included 
instantaneous reflection, absorption, and the elastic Robin condition.
Wentzell \cite{venttsel1959boundary} extended the result to more general domains.

It\^{o} and McKean \cite{ito1963brownian} constructed sample paths to the
one-dimensional sticky reflected Brownian motion
\begin{equation}
\label{sticky_first}
 dX_t = 2\mu 1_{\{X_t = 0\}}dt + 1_{\{X_t>0\}}dW_t, \quad \mu > 0,
\end{equation}
whose infinitesimal generator is the one studied by 
Feller. Skorokhod conjectured that the sticky reflected Brownian motion has no strong solution.
A proof that \eqref{sticky_first} has a unique weak solution 
can be found in for example \cite[IV.7]{watanabe1981stochastic}.

Chitashvili published the technical report \cite{chitashvili1989nonexistence} 
in 1989 claiming a proof of Skorokhod's conjecture.
Around that time, the process was studied by several authors, e.g. 
\cite{harrison1981sticky, graham1988martingale, amir1991sticky, yamada1994reflecting}, 
to name a few.
Warren \cite{warren1997branching} provided a proof of Skorokhod's conjecture in
1997 and in 2014 Engelbert and Peskir \cite{engelbert2014stochastic} 
published a proof useful for further generalizations.
The fact that the system has no strong 
solution has consequences for how optimal control of the system can be 
approached, as we will see in this paper.

Building on \cite{engelbert2014stochastic},
interacting particle systems of sticky reflected Brownian 
motions are considered in \cite{grothaus2018strong}. Interaction is introduced 
via a Girsanov transformation. See \cite[Sect. 3.2]{grothaus2018strong} for the 
construction. Under assumptions on the 'shape' of the interaction and 
integrability of the Girsanov kernel, the interacting system is well-defined. 
Since the process no longer behaves like a Brownian motion in the interior of the domain,
it is now referred to as a sticky reflected SDE.
The boundary behavior is shown to be sticky in the sense that the process spends 
a ($dt$-)positive time on the boundary.

Sticky reflected SDEs with boundary diffusion are considered in  
\cite{grothaus2017stochastic}. The paths defined by such a system are allowed to 
move on the (sufficiently smooth) boundary $\partial\mathcal{D}$ of some bounded domain 
$\mathcal{D} \subset \mathbb{R}^d$. Under smoothness conditions on $\partial\mathcal{D}$, the 
authors show that this type of SDE has a unique weak solution. 
Furthermore, an interacting system is studied, where interaction is introduced 
via a Girsanov transformation.

\subsection{Synopsis}

In this paper, the sticky reflected SDE with boundary diffusion of 
\cite{grothaus2017stochastic} is proposed as a model for pedestrian crowd motion 
in confined domains. 
We begin by considering a (non-transformed) sticky reflected SDE with boundary 
diffusion on $\mathcal{D}$, a non-empty bounded subset of $\mathbb{R}^n$ with 
$C^2$-smooth boundary $\Gamma := \partial \mathcal{D}$ (see Section 
\ref{sec:bg-bd}, below)  and outward normal $n$,
\begin{equation}
\label{eq:intro_eqsys_1}
    dX_t = \left(1_{\mathcal{D}}(X_t) + 1_\Gamma(X_t)\pi(X_t)\right)dB_t 
    - 1_\Gamma(X_t)\frac{1}{2}\Big(\frac{1}{\gamma} + \kappa(X_t)\Big)n(X_t)dt,
\end{equation}
where $\pi(X_t)$ is the projection onto the tangent space of $\Gamma$ at $X_t$, 
$\kappa(X_t)$ is the mean curvature of $\Gamma$ at $X_t$, and $\gamma$ is a positive constant 
representing the stickiness of $\Gamma$, cf. Remark~\ref{remark:gamma} in 
Section~\ref{sec:model} below. All relevant technical details can be found in 
Section~\ref{sec:prels}. 
Equation \eqref{eq:intro_eqsys_1} admits a unique weak solution $\mathbb{P}$, 
but no strong solution. To control an equation that admits only a weak solution 
is to control a probability measure on $(\Omega,\mathcal{F})$, under which the 
state process $X_\cdot := \{X_t\}_{t\in[0,T]}$ is interpreted as the coordinate 
process $X_t(\omega) = \omega(t)$.
If all the admissible distributions of $X_\cdot$ are absolutely continuous 
with respect the reference measure $\mathbb{P}$, then Girsanov's
theorem can be used to implement the control. This corresponds to for the case
when the drift of \eqref{eq:intro_eqsys_1} is controlled. 
In the controlled diffusion case, admissible measures are all 
singular with $\mathbb{P}$ and with one another (for different controls), and 
the control problem is in fact a robustness problem over all admissible measures 
which leads to the so-called second order backward SDE framework 
\cite{soner2012wellposedness}. In this paper we treat the case with controlled 
drift, the controlled diffusion case will be treated elsewhere. 
A mean-field dependent drift $\beta$ is introduced into the coordinate process 
through the Girsanov transformation
\begin{equation}
\frac{d\mathbb{P}^u}{d\mathbb{P}}{\Big|}_{\mathcal{F}_t}= L_t^u := 
\mathcal{E}_t\left(\int_0^\cdot \beta \left(t, X_\cdot, \mathbb{P}^u(t), 
u_t\right)^*dB_t\right),
\end{equation}
where $\mathbb{P}^u(t):=\mathbb{P}^u\circ X^{-1}_t$ is the marginal distribution 
of $X_t$ under $\mathbb{P}^u$, $\beta^*$ denotes the transpose of $\beta$, and 
$\mathcal{E}$ is the Dol\'eans-Dade 
exponential defined for a continuous local martingale $M$ as 
\begin{equation}
    \label{eq:doleans}
    \mathcal{E}_t(M):=\exp{\left(M_t-\frac{1}{2}\langle M \rangle_t\right)}.
\end{equation}
The path of a typical pedestrian in the interacting crowd is then (under 
$\mathbb{P}^u$) described by
\begin{equation}
\label{eq:intro_eqsys_3}
\left\{
 \begin{aligned}
  &dX_t = 1_{\mathcal{D}}(X_t)\Big(\beta\left(t, X_\cdot, \mathbb{P}^u(t), 
u_t\right)dt 
+ dB^u_t\Big)
  \\
  &\hspace{1cm}
  + 1_\Gamma(X_t)\left(\pi(X_t)\beta\left(t, X_\cdot, \mathbb{P}^u(t), 
u_t\right) - 
  \frac{n(X_t)}{2\gamma}\right)dt
  \\
  &\hspace{1cm}
  + 1_\Gamma(X_t)dB_t^{\Gamma,u},
  \\
  &dB_t^{\Gamma,u} = \pi(X_t) dB^u_t - \frac{1}{2}\kappa(X_t)n(X_t)dt,
\end{aligned}
\right.
\end{equation}
where $B^u$ is a $\mathbb{P}^u$-Brownian motion. 
We provide a proof of the existence of the controlled probability measure  
$\mathbb{P}^u$ based on a fixed-point argument involving the total variation 
distance (cf. \cite{djehiche2018optimal}). 

Pedestrians are assumed to be cooperating and controlled by a rational central 
planner. The central planner represents an authority that gives directions to 
the crowd through signs, mobile devices, or security personnel, and the crowd follows 
the instructions. This setup has been used to study evacuation 
in for example \cite{burger2013mean2, burger2014mean, djehiche2017mean}.
For a discussion on the goals, the degrees of cooperation, and the information structure
in a pedestrian crowd, see \cite{cristiani2015modeling}. 
The central planner's goal is to minimize the finite-horizon 
cost functional
\begin{equation}
\label{eq:intro_cost}
 J(u) 
 := 
 E^u\left[\int_0^T  f\left(t, X_\cdot, \mathbb{P}^u(t), u_t\right)dt 
 +  g\left(X_T, \mathbb{P}^u(T)\right)\right],
\end{equation}
where $f$ is the instantaneous cost and $g$ is the terminal cost (see Section~\ref{sec:control}
for conditions on the functions $f$ and $g$).
The minimization of \eqref{eq:intro_cost} subject to \eqref{eq:intro_eqsys_3} is 
equivalent to the following mean-field type control problem, stated in the 
strong sense in the original probability space with measure $\mathbb{P}$,
\begin{equation}
\label{eq:intro_opt}
\left\{
\begin{aligned}
 &\inf_{u \in\mathcal{U}}\ E\left[\int_0^T L_t^u f\left(t, X_\cdot, 
\mathbb{P}^u(t), u_t\right)dt + L_T^u g\left(X_T, \mathbb{P}^u(T)\right)\right],
 \\
 &\ \text{s.t.}\ \hspace{4pt} dL_t^u  =  L_t^u\beta\left(t, X_\cdot, 
\mathbb{P}^u(t), u_t\right)^*dB_t,\hspace{4pt} L_0^u = 1.
\end{aligned}
\right.
\end{equation}
The validity of \eqref{eq:intro_opt} is justified in Section~\ref{sec:control} below.
Problem \eqref{eq:intro_opt} is nowadays a standard mean-field type control 
problem and a stochastic maximum principle yielding necessary conditions for an 
optimal control can be found in \cite{buckdahn2016stochastic}. Solving the 
general problem \eqref{eq:intro_opt} with a Pontryagin-type maximum principle 
poses some practical difficulties, the main one being the necessity of a second 
order adjoint process. However, most difficulties can be tackled by imposing 
assumptions plausible for the application in pedestrian crowd motion. With the 
aim to replicate the pedestrian behavior observed in the empirical studies 
\cite{zanlungo2012microscopic} and \cite{zhang2012pedestrian}, we consider here 
a special case of \eqref{eq:intro_opt} where $u_t$ takes values in a 
convex set and $\mathbb{P}^u(t)$ is replaced by 
$E^u[r(X_t)]$, where the function $r: \mathbb{R}^d\rightarrow\mathbb{R}^d$ can be 
different for each of the coefficients involved.

\subsection{Paper contribution and outline}

The main contribution of this paper is a new approach to boundary conditions in 
pedestrian crowd modeling. Sticky reflected SDEs of mean-field type with 
boundary diffusion is proposed as an alternative to reflected SDEs of mean-field 
type to model pedestrian paths in optimal-control based models. Sticky 
boundaries and boundary diffusion allow the pedestrian to spend time and move 
along the boundary (walls, pillars, etc.), in contrast to reflected SDE-based 
models where pedestrians are immediately reflected. Existence and uniqueness of 
the mean-field type version of the sticky reflected SDE with boundary diffusion 
is treated. The model can be optimally controlled (in the weak sense) and a 
Pontryagin-type stochastic maximum principle is applied to derive necessary 
optimality conditions. Furthermore, the mean-field type control problem has a 
microscopic interpretation in the form of a system of interacting sticky 
reflected SDEs with boundary diffusion.
The new features of sticky boundaries and boundary diffusion yield more 
flexibility when modeling pedestrian behavior at boundaries. A scenario of 
unidirectional pedestrian flow in a long narrow corridor is studied numerically 
to highlight these novel characteristics and to replicate experimental findings 
as a first step in model validation.

The rest of the paper is organized as follows. Section~\ref{sec:prels} defines 
notation and summarizes relevant background theory. Section~\ref{sec:model} 
introduces sticky reflected SDEs of mean-field type with boundary diffusion. 
Conditions under which the equation has a unique weak solution are presented. In 
Section~\ref{sec:control} the finite horizon optimal control of the state 
equation introduced in Section~\ref{sec:model} is considered. In the uncontrolled case, the convergence on an interacting (non-mean-field) 
particle system to the sticky reflected SDE of mean-field type is proved. Finally, Section~\ref{sec:ex} presents analytic  examples and numerical results based on the particle system approximation concerning unidirectional flow in a long narrow corridor.

\section{Preliminaries}
\label{sec:prels}
 The domain $\mathcal{D}$ is a non-empty bounded subset of $\mathbb{R}^d$ with 
$C^2$-smooth boundary $\Gamma := \partial \mathcal{D}$. The closure of 
$\mathcal{D}$ is denoted $\bar{\mathcal{D}}$.  The Euclidean norm  
is denoted $|\cdot|$. A finite 
time horizon $T>0$ is fixed throughout the paper. The path of a stochastic 
process is denoted $X_\cdot := \{X_t\}_{t\in[0,T]}$, and $C$ is a generic 
positive constant.
 
\subsection{The coordinate process and probability metrics}
\label{sec:bg-metrics}

Let $(\mathcal{X},d)$ be a metric space.
The set of Borel probability measures on $\mathcal{X}$ is 
denoted by $\mathcal{P}(\mathcal{X})$. 
By $\mathcal{P}_p(\mathcal{X}) \subset \mathcal{P}(\mathcal{X})$ we denote the set of all
$\mu\in\mathcal{P}(\mathcal{X})$ such that $(\|\mu\|_p)^p := \int d(y_0,y)^p \mu(dy)
< \infty$ for an arbitrary $y_0\in\mathcal{X}$.

Let $\Omega := C([0,T]; \mathbb{R}^d)$ be endowed with the metric 
$|\omega|_T := \sup_{t\in[0,T]}|\omega(t)|$ for $\omega\in\Omega$. Denote by 
$\mathcal{F}$ the Borel $\sigma$-field over $\Omega$. Given $t\in[0,T]$ and 
$\omega\in\Omega$, put $X_t(\omega) = \omega(t)$ and denote by $\mathcal{F}^0_t 
:= \sigma(X_s; s\leq t)$ the filtration generated by $X_\cdot$. $X_\cdot$ is the 
so-called \textit{coordinate process}. For any $P\in\mathcal{P}(\Omega)$ (the 
set of Borel probability measures on $\Omega$) we denote by $\mathbb{F}^P := 
(\mathcal{F}^P_t; t\in[0,T])$ the completion of $\mathbb{F}^0 := (\mathcal{F}^0_t; 
t\in[0,T])$ with the $P$-null sets of $\Omega$. 

Let $\mu,\nu \in \mathcal{P}(\mathbb{R}^d)$ and let $\mathcal{B}(\mathbb{R}^d)$ 
be the Borel $\sigma$-algebra on $\mathbb{R}^d$. The total variation metric on 
$(\mathbb{R}^d, \mathcal{B}(\mathbb{R}^d))$ is 
\begin{equation}
d_{TV}(\mu,\nu) := 
2\hspace{-.2cm}\sup_{A\in\mathcal{B}(\mathbb{R}^d)}\hspace{-.2cm}\left|\mu(A) - 
\nu(A)\right|.    
\end{equation}
On the filtration $\mathbb{F}^P$, where $P \in \mathcal{P}(\Omega)$,
the total variation metric between $m,m'\in\mathcal{P}(\Omega)$ is 
\begin{equation}
    D_t(m,m') := 2 \sup_{A\in\mathcal{F}^P_t}\left|m(A) - m'(A)\right|,\quad 0\leq 
t\leq T,,
\end{equation} 
and satisfies $D_s(m,m') \leq D_t(m,m')$ for $0\leq s\leq t$. Consider the 
coordinate process $X_\cdot$, then for $m,m' \in \mathcal{P}(\Omega)$,
\begin{equation}
    d_{TV}\left(m \circ X_t^{-1}, m' \circ X_t^{-1}\right)\leq D_t(m,m'),\quad 
0\leq t\leq T.
\end{equation}
Endowed with the metric $D_T$, $\mathcal{P}(\Omega)$ is a complete metric 
space. 
The total variation metric is connected to the Kullback-Leibler  divergence 
through the Csisz\'ar-Kullback-Pinsker inequality,
\begin{equation}
\label{eq:ckp_ineq}
    D^2_t(m, m') \leq 2E^m\left[\log\left(dm/dm'\right)\right],
\end{equation}
where $E^m$ denotes expectation with respect to $m$.


\subsection{Boundary diffusion}
\label{sec:bg-bd}
In this subsection we introduce the boundary diffusion $B^\Gamma$ and 
review the necessary parts of the background theory presented in \cite[Sect. 
2]{grothaus2017stochastic}. 
 
\begin{definition}
\label{def:lipschitz_bry}
$\Gamma$ is \textit{Lipschitz continuous} (resp. \textit{$C^k$-smooth}) if for 
every $x\in \Gamma$ there exists a neighborhood $V\subset \mathbb{R}^d$ of $x$ 
such that $\Gamma \cap V$ is the graph of a Lipschitz continuous (resp. 
$C^k$-smooth) function and 
$\mathcal{D}\cap V$ is located on one side of the graph, i.e., there exists new 
orthogonal coordinates $(y_1,\dots, y_d)$ given by an orthogonal map $T$, a 
reference point $z \in \mathbb{R}^{d-1}$, real numbers $r,h > 0$, and a 
Lipschitz continuous (resp. $C^k$-smooth) function $\varphi : 
\mathbb{R}^{d-1}\rightarrow \mathbb{R}$ such that
\begin{itemize}
 \item[(i)]  $V = \{ y \in \mathbb{R}^d : |y_{-d} - z| < r, |y_d - 
\varphi(y_{-d})| < h\}$
  \item[(ii)] $\mathcal{D} \cap V = \{y \in V : -h < y_d - \varphi(y_{-d}) < 
0\}$
\item[iii)] $\Gamma \cap V = \{y \in V : y_d = \varphi(y_{-d})\}$
\end{itemize}
\end{definition}
\begin{definition}
\label{def:out_normal}
 For $y \in V$, let
\begin{equation}
 \widetilde{n}(y) := \frac{ \left( - \nabla 
\varphi(y_{-d}), 1\right)}{\sqrt{|\nabla\varphi(y_{-d})|^2 + 1}}.
\end{equation}
Let $x\in\Gamma$ and $T\in 
\mathbb{R}^{d\times d}$ be the orthogonal transformation from 
Definition~\ref{def:lipschitz_bry}. Then the \textit{outward normal vector} at 
$x$ is defined by $n(x) := T^{-1}\widetilde{n}(Tx)$.
\end{definition}
\begin{definition}
Let $x\in\Gamma$ and $\pi(x) := E - n(x)n(x)^* \in 
\mathbb{R}^{d\times d}$, where $E$ is the identity matrix. $\pi(x)$ is the 
\textit{orthogonal projection} on the tangent space at $x$.
\end{definition}
Note that for 
$z\in\mathbb{R}^d$, $\pi(x) z = z - \left( n(x),z \right)n(x)$.
\begin{definition}
 Let $f\in C^1(\bar{\mathcal{D}})$ and $x\in \Gamma$. Whenever $\Gamma$ is 
sufficiently smooth at $x$,  $\nabla_\Gamma f(x) := \pi(x)\nabla f(x)$
and if $f\in C^2(\bar{\mathcal{D}})$, $\Delta_\Gamma f(x) := 
Tr(\nabla_\Gamma^2f(x))$.
If $n$ is differentiable at $x$ the \textit{mean curvature} of $\Gamma$ at $x$ 
is
\begin{equation}
 \kappa(x) := div_\Gamma n(x) = \left(\pi(x)\nabla\right) \cdot n(x).
\end{equation}
\end{definition}
In \cite{grothaus2017stochastic} it is noted that whenever $\Gamma$ is 
$C^2$-smooth, 
\begin{equation}
    \left(\pi\nabla\right)^*\pi = - \kappa n.
\end{equation}
A Brownian motion $B^\Gamma_\cdot$ on a smooth boundary $\Gamma$ is a 
$\Gamma$-valued stochastic process generated by $\frac{1}{2}\Delta_\Gamma$. 
This is in analogy with the standard Brownian motion on $\mathbb{R}^d$, in the 
sense that $B^\Gamma_\cdot$ solves the martingale problem for 
$(\frac{1}{2}\Delta_\Gamma, C^\infty(\Gamma))$. A solution to the 
Stratonovich SDE
\begin{equation}
 dB^\Gamma_t = \pi(B^\Gamma_t)\circ dB_t,
\end{equation}
where $B_\cdot$ is a standard Brownian motion on $\mathbb{R}^d$, is a Brownian 
motion on $\Gamma$ \cite[Chap. 3, Sect. 2]{hsu2002stochastic}. 
By the It{\^o}-Stratonovich transformation rule, the Brownian motion on $\Gamma$ 
solves
\begin{equation}
 dB_t^\Gamma 
  = -\frac{1}{2}\kappa(B^\Gamma_t) n(B^\Gamma_t)dt + \pi(B^\Gamma_t)dB_t.
\end{equation}

\section{Sticky reflected SDEs of mean-field type with boundary diffusion}
\label{sec:model}

In this section we provide conditions for the existence and uniqueness of a weak 
solution to the sticky reflected SDE of mean-field type with boundary diffusion.
Consider the reflected sticky SDE with boundary diffusion,
\begin{equation}
\label{eq:non_trans_state}
\left\{
 \begin{aligned}
  dX_t 
 &=
 - 1_\Gamma(X_t)\frac{1}{2}\left(\frac{1}{\gamma} + 
\kappa(X_t)\right)n(X_t)dt 
\\ 
&
\hspace{30pt} + 
 \left(1_{\mathcal{D}}(X_t) + 1_\Gamma(X_t)\pi(X_t)\right)dB_t,
 \\
 X_0 &= x_0 \in \mathcal{\bar{D}},
 \end{aligned}
 \right.
\end{equation}
which from now on will be written in short-hand notation as
\begin{equation}
\label{eq:abbreviation}
 dX_t 
 = 
 a(X_t)dt + \sigma(X_t)dB_t,
\end{equation}
where $a : [0,T]\times\mathbb{R}^d \rightarrow \mathbb{R}^d$ and $\sigma : 
[0,T]\times\mathbb{R}^d \rightarrow\mathbb{R}^{d\times d}$ are bounded functions 
over $[0,T]\times\bar{\mathcal{D}}$, defined as
\begin{equation}
    a(x) := -1_\Gamma(x)\frac{1}{2}\left(\frac{1}{\gamma} + 
\kappa(x)\right)n(x),\hspace{6pt} \sigma(x) := 1_{\mathcal{D}}(x) + 
1_\Gamma(x)\pi(x).
\end{equation}
By \cite[Thm 3.9 \& 3.17]{grothaus2017stochastic}, \eqref{eq:non_trans_state} has a unique
weak solution, i.e. there is a unique
probability measure $\mathbb{P}$ on $(\Omega, \mathcal{F})$ that solves the 
corresponding martingale problem (cf. \cite[Thm 18.7]{kallenberg}), 
and the solution $X_\cdot$ is $C([0,T];\bar{\mathcal{D}})$-valued $\mathbb{P}$-a.s.
The result \cite[Thm 3.9]{grothaus2017stochastic} relies on some
conditions, lets verify them for the sake of 
completeness. The weight functions $\alpha$ and $\beta$, introduced on \cite[pp. 6]{grothaus2017stochastic},
are in \eqref{eq:non_trans_state} set to be everywhere constant and positive 
such that $\alpha/\beta = 1/\gamma$ (cf. Remark~\ref{remark:gamma}, below). 
Condition 3.12 of \cite{grothaus2017stochastic} therefore holds: $\partial\mathcal{D}$ is
$C^2$ and the constant positive weight functions have the required regularity.
This justifies the use of \cite[Thm 3.9]{grothaus2017stochastic}, and no further conditions are required for
\cite[Thm 3.17]{grothaus2017stochastic}. To simplify notation, from now on through out the rest of
this paper let $\mathbb{F}$ denote the completion of $\mathbb{F}^0$ with the $\mathbb{P}$-null sets
of $\Omega$, i.e.
$\mathbb{F} = (\mathcal{F}_t; t\geq 0) := \mathbb{F}^{\mathbb{P}}$.

\begin{remark}
\label{remark:gamma}
The coordinate process is composed of three essential parts:
\begin{itemize}
    \item Interior diffusion $1_{\mathcal{D}}(X_t)dB_t$;
    \item Boundary diffusion $1_\Gamma(X_t)(\pi(X_t)dB_t - \frac{1}{2}(\kappa 
n)(X_t)dt) = 1_\Gamma(X_t)dB_t^\Gamma$;
    \item Normal sticky reflection $-1_\Gamma(X_t)\frac{1}{2\gamma}n(X_t)dt$.
\end{itemize}
The constant $\gamma$ is connected to the level of stickiness of the boundary 
$\Gamma$. It is related to the invariant distribution of the coordinate 
processes' $\mathbb{R}^d$-valued time marginal. Let $\lambda$ and $s$ denote the 
Lebesgue measure on $\mathbb{R}^d$ and the surface measure on $\Gamma$, 
respectively. Consider the measure $\rho := 1_{\mathcal{D}}\alpha\lambda + 
1_{\Gamma}\alpha' s$, $\alpha, \alpha' \in \mathbb{R}$. By choosing $\alpha = 
\bar{\alpha}/\lambda(\mathcal{D})$ and $\alpha' = (1-\bar{\alpha})/s(\Gamma)$, 
$\bar{\alpha}\in[0,1]$, $\rho$ becomes a probability measure on $\mathbb{R}^d$ 
with support in $\bar{\mathcal{D}}$ and $\rho$ is in fact the invariant 
distribution of \eqref{eq:non_trans_state} whenever
\begin{equation}
    \frac{1}{\gamma} = 
\frac{\bar{\alpha}}{(1-\bar{\alpha})}\frac{s(\Gamma)}{\lambda(\mathcal{D})}.
\end{equation}
Hence $\bar{\alpha}\rightarrow 1$ as $\gamma \rightarrow 0$ and the invariant 
distribution of \eqref{eq:non_trans_state} concentrates on the interior 
$\mathcal{D}$. But as $\gamma \rightarrow\infty$, it concentrates on the 
boundary $\Gamma$. We say that the more probability mass that  $\rho$ locates on 
$\Gamma$, the stickier $\Gamma$ is.
\end{remark}
Next, we introduce mean-field interactions and a control process in 
\eqref{eq:non_trans_state} through a Girsanov transformation.
\begin{definition}
Let the set of control values $U$ be a subset of $\mathbb{R}^d$. The 
\textit{set of admissible controls} is
\begin{equation}
\label{eq:def_calU}
    \mathcal{U} 
    := 
    \left\{ u : [0,T] \times \Omega \rightarrow U : u\ \mathbb{F}\text{-prog. 
measurable} \right\}.
\end{equation}
\end{definition}
Let ${\mathbb{Q}}(t) := {\mathbb{Q}} \circ X^{-1}_t$ denote the $t$-marginal 
distribution of the coordinate process under 
${\mathbb{Q}}\in\mathcal{P}(\Omega)$.
Let $\beta$ be a measurable function from $[0,T]\times 
\Omega\times\mathcal{P}(\mathbb{R}^d)\times U$ into $\mathbb{R}^d$ such that
\begin{assumption}
\label{assump:beta_meas}
   For every $\mathbb{Q}\in\mathcal{P}(\Omega)$ and $u\in\mathcal{U}$,
   $\left(\beta(t, X_\cdot, \mathbb{Q}(t), u_t)\right)_{t\in[0,T]}$ 
is progressively measurable with respect to $\mathbb{F}$, the 
completion of the filtration generated by the coordinate process with the $\mathbb{P}$-null 
sets of $\Omega$.
 \end{assumption}
 \begin{assumption}
\label{assump:beta_lin}
 For every $t\in[0,T]$, $\omega\in\Omega$, $u\in U$, and $\mu\in 
\mathcal{P}(\mathbb{R}^d)$,
 \begin{equation}
     |\beta(t,\omega,\mu,u)| \leq C\left(1+|\omega|_T + \int_{\mathbb{R}^d} 
|y|\mu(dy)\right).
 \end{equation}
\end{assumption} 
\begin{assumption}
\label{assump:beta_lip}
    For every $t\in[0,T]$, $\omega \in \Omega$, $u\in U$, and $\mu,\mu' \in 
\mathcal{P}(\mathbb{R}^d)$,
    \begin{equation}
        \left|\beta\left(t, \omega, \mu, u\right)-\beta\left(t, \omega, \mu', 
u\right)\right| 
        \leq 
        Cd_{TV}(\mu,\mu').
    \end{equation}
\end{assumption}
Given ${\mathbb{Q}}\in \mathcal{P}(\Omega)$ and $u\in\mathcal{U}$, let
\begin{equation}
\label{eq:likelihood-def}
L^{u,{\mathbb{Q}}}_t := \mathcal{E}_t\left(\int_0^\cdot \beta\left(s,X_\cdot, 
{\mathbb{Q}}(s), u_s\right)dB_s\right), 
\end{equation} 
where $\mathcal{E}$ is the Dol\'eans-Dade 
exponential (cf. \eqref{eq:doleans}).
\begin{lemma}
\label{eq:lemma_1}
The positive measure  $\mathbb{P}^{u,{\mathbb{Q}}}$ defined by  
$d\mathbb{P}^{u,{\mathbb{Q}}} = L_t^{u,\mathbb{Q}} d\mathbb{P}$ on 
$\mathcal{F}_t$ for all $t\in[0,T]$, is well-defined and is a probability measure on $\Omega$. 
Moreover, $\mathbb{P}^{u,\mathbb{Q}}\in\mathcal{P}_p(\Omega)$ for all $p\in[1,\infty)$ and 
under $\mathbb{P}^{u,{\mathbb{Q}}}$ the coordinate process satisfies
\begin{equation}
\label{eq:lemma_ch_3}
 X_t = x_0 + \int_0^t \Big(\sigma(X_s)\beta\left(s,X_\cdot, \mathbb{Q}(s), 
u_s\right)  + a(X_s)\Big)ds
 +
 \int_0^t \sigma(X_s)dB^{\mathbb{Q}}_s,
\end{equation}
where $B^{\mathbb{Q}}$ is a standard $\mathbb{P}^{u,{\mathbb{Q}}}$-Brownian 
motion.
\end{lemma}
\begin{proof}
Assume that $\varphi_\cdot$ is a process such that $\mathbb{P}^\varphi$, defined 
by $d\mathbb{P}^\varphi = L_t^\varphi d\mathbb{P}$ on $\mathbb{F}_t$ where 
$L^\varphi_t := \mathcal{E}_t(\int_0^\cdot \varphi_s dB_s)$, is a probability 
measure on $\Omega$. By Girsanov's theorem, the coordinate process under 
$\mathbb{P}^\varphi$ satisfies
\begin{equation}
    dX_t = \left(\sigma(X_t)\varphi_t + a(X_t)\right)dt + 
\sigma(X_t)dB^\varphi_t,
\end{equation}
where $B^\varphi_\cdot$ is a $\mathbb{P}^\varphi$-Brownian motion. 
$C^2$-smoothness of the boundary $\Gamma$ grants a bounded orthogonal projection 
on $\Gamma$'s tangent space and a bounded mean curvature of $\Gamma$. By the 
Burkholder-Davis-Gundy inequality we have for $1 \leq p < \infty$
\begin{equation}
\label{eq:crucial_bound}
\begin{aligned}
    E^\varphi\left[|X|_T^p\right] 
    &\leq 
    E^\varphi\Bigg[C\Bigg(|X_0|^p + \int_0^T |\sigma(X_s)\varphi_s|^p ds 
    + 
\int_0^T
|a(X_s)|^p ds 
    \\
    &\hspace{1cm}
    + \left|\int_0^\cdot \sigma(X_s)dB^\varphi_s \right|_T^p \Bigg)\Bigg]
    \\
    &\leq
    C\left(1 + \int_0^TE^\varphi[|\varphi_s|^p]ds\right),
\end{aligned}
\end{equation}
where $E^\varphi$ denotes expectation taken under $\mathbb{P}^\varphi$. 
By Assumption~\ref{assump:beta_lip} it holds for every $t\in[0,T]$, $\omega \in 
\Omega$, $\mu\in\mathcal{P}(\mathbb{R}^d)$, and $u\in U$ that
\begin{equation}
\label{eq:aux_lemma1}
\begin{aligned}
    \left|\beta\left(t, \omega, \mu, u\right)\right| 
    &\leq 
C\Big(d_{TV}(\mu,\mathbb{P}(t)) + 
\left|\beta\left(t,\omega,\mathbb{P}(t),u\right)\right|\Big).
\end{aligned}
\end{equation}
In view of \eqref{eq:aux_lemma1}, 
Assumption~\ref{assump:beta_lin} and \ref{assump:beta_lip}, and the fact that the total
variation between two probability measures is uniformly bounded, we have
for all $t\in[0,T]$, 
\begin{equation}
 \label{eq:betaineqality0}
 \begin{aligned}
  |\beta(t,X_\cdot, \mathbb{Q}(t), u_t)| 
  &\leq 
  C\left(d_{TV}(\mathbb{Q}(t), \mathbb{P}(t)) 
  + |\beta(t, X_\cdot, \mathbb{P}(t), u_t)|\right)
  \\
  & \leq C\left(1 + |X|_T 
  + \int_{\mathbb{R}^d} |y|\mathbb{P}(t)(dy)\right) 
  \\
  &\leq C\left(\sup\{ |y|: y\in\bar{\mathcal{D}}\}\right) 
  =: \bar{C} < \infty, \quad \mathbb{P}\text{-a.s.}
 \end{aligned}
\end{equation}
The third inequality of \eqref{eq:betaineqality0} holds $\mathbb{P}$-a.s. since 
under $\mathbb{P}$, 
$X_\cdot \in C\left([0,T]; \bar{\mathcal{D}}\right)$ almost surely.
We note that \eqref{eq:betaineqality0} implies that Novikov's condition is satisfied,
\begin{equation}
  E\left[\exp\left(\frac{1}{2}\int_0^T\sup_{s\in[0,T]}|\beta(s,X_\cdot, \mathbb{Q}(s), u_s)|^2dt 
  \right) \right] 
  \leq 
  E\left[\exp\left(\frac{T\bar{C}^2}{2}\right) \right] 
  < 
  \infty,
\end{equation}
where $E$ denotes expectation with respect to $\mathbb{P}$.
Hence the Doléans-Dade exponential defined 
in \eqref{eq:likelihood-def} is an $(\mathcal{F}_t, \mathbb{P})$-martingale and
$\mathbb{P}^{u,\mathbb{Q}}$ is indeed a probability measure, 
i.e. $\mathbb{P}^{u,\mathbb{Q}}\in\mathcal{P}(\Omega)$.
To show that $\mathbb{P}^{u,\mathbb{Q}}\in\mathcal{P}_p(\Omega)$ for any $p\in[1,\infty)$, 
we simply note that
\begin{equation}
\label{eq:girsanov_change_}
\begin{aligned}
       &E^{u,\mathbb{Q}}\left[|X|_T^p\right] 
       \\
       &=
       E^{u,\mathbb{Q}}\left[|X|_T^p
       \left(1_{\{X_\cdot \in C([0,T]; \bar{\mathcal{D}})\}}
       + 
       1_{\{X_\cdot \notin C([0,T]; \bar{\mathcal{D}})\}}
       \right)\right] 
       \\
       &=
       \mathbb{E}\left[L^{u,\mathbb{Q}}_T|X|_T^p
       \left(1_{\{X_\cdot \in C([0,T]; \bar{\mathcal{D}})\}}
       + 
       1_{\{X_\cdot \notin C([0,T]; \bar{\mathcal{D}})\}}
       \right)\right] 
       \\
       &\leq
       \sup\{|y|^p: y\in\bar{\mathcal{D}}\}
       \mathbb{E}\left[L^{u,\mathbb{Q}}_T1_{\{X_\cdot \in C([0,T]; \bar{\mathcal{D}})\}}\right]
       \\
       &=
       \sup\{|y|^p: y\in\bar{\mathcal{D}}\}.
\end{aligned}
\end{equation}
Finally, by Girsanov's theorem the coordinate process 
under $\mathbb{P}^{u,\mathbb{Q}}$ satisfies \eqref{eq:lemma_ch_3}.
\end{proof}
For a given $u \in\mathcal{U}$, consider the map
\begin{equation}
    \Phi_u : \mathcal{P}(\Omega) \ni \mathbb{{\mathbb{Q}}} \mapsto 
\mathbb{P}^{u,\mathbb{{\mathbb{Q}}}} \in \mathcal{P}(\Omega),
\end{equation}
such that $d\mathbb{P}^{u,\mathbb{Q}} = L^{u,\mathbb{Q}}_td\mathbb{P}$ on $\mathcal{F}_t$, where $L^{u,\mathbb{Q}}$ is given by 
\eqref{eq:likelihood-def}.
\begin{proposition}
\label{prop:unique_fixed_point}
The map $\Phi_u$ is well-defined and admits a unique fixed point for all $u\in \mathcal{U}$. 
Moreover, for every $p\in [1,\infty)$ the fixed point, denoted $\mathbb{P}^u$, belongs to 
$\mathcal{P}_p(\Omega)$. In particular,
 \begin{equation}
     E^u\left[\left|X\right|_T^p\right] 
     \leq
     \sup_{y\in\bar{\mathcal{D}}} |y|^p,
 \end{equation}
where $E^u$ denotes expectation with respect to $\mathbb{P}^u$.
\end{proposition}
\begin{proof}
By Lemma~\ref{eq:lemma_1}, the mapping is well defined. We first show the contraction 
property of the map $\Phi_u$ in the complete metric space $\mathcal{P}(\Omega)$, 
endowed with the total variation distance $D_T$. The proof is an adaptation of 
the proof of \cite[Thm. 8]{choutri2016optimal}. For each $t\in[0,T]$, let 
$\beta^{\mathbb{Q}}_t := \beta(t, X_\cdot, \mathbb{Q}(t), u_t)$. Given  
${\mathbb{Q}},\widetilde{{\mathbb{Q}}}\in\mathcal{P}(\Omega)$, the 
Csisz\'ar-Kullback-Pinsker inequality \eqref{eq:ckp_ineq} and the fact that 
$\int_0^\cdot (dB_s - \beta^{\mathbb{Q}}_sds)$ is a martingale under 
$\Phi_u({\mathbb{Q}}) = \mathbb{P}^{u,\mathbb{Q}}$ yields
\begin{equation}
\label{eq:fixed_point_estimate}
 \begin{aligned}
  &D_T^2\left(\Phi_u({\mathbb{Q}}), \Phi_u(\widetilde{{\mathbb{Q}}})\right)
  \leq
2E^{u,\mathbb{Q}}\left[\log\left(L^{u,\mathbb{Q}}_T/L^{u,\widetilde{{\mathbb{Q
}}}}_T\right)\right]
  \\
  &= 
  2E^{u,\mathbb{Q}}\left[
  \int_0^T 
  \left(\beta^{\mathbb{Q}}_s - \beta^{\widetilde{{\mathbb{Q}}}}_s\right)dB_s
  -  \frac{1}{2} \int_0^T \Big(\beta^{\mathbb{Q}}_s\Big)^2 - 
\Big(\beta^{\widetilde{{\mathbb{Q}}}}_s\Big)^2 ds\right]
  \\
  &= 
  2E^{u,\mathbb{Q}}\left[  \int_0^T \left(\beta^{\mathbb{Q}}_s - 
    \beta^{\widetilde{{\mathbb{Q}}}}_s\right)\beta^{\mathbb{Q}}_s - 
\frac{1}{2}\Big(\beta^{\mathbb{Q}}_s\Big)^2 + 
\frac{1}{2}\Big(\beta^{\widetilde{{\mathbb{Q}}}}_s\Big)^2 ds \right]
  \\
  &=
  \int_0^T \mathbb{E}^{u,\mathbb{Q}}\left[\left(\beta^{\mathbb{Q}}_s - 
\beta^{\widetilde{{\mathbb{Q}}}}_s\right)^2\right]ds
  \\
  &\leq C\int_0^T d_{TV}^2\left({\mathbb{Q}}(s), 
\widetilde{{\mathbb{Q}}}(s)\right)ds
  \leq 
  C \int_0^T D_s^2\left({\mathbb{Q}},\widetilde{{\mathbb{Q}}}\right)ds.
 \end{aligned}
\end{equation}
Iterating the inequality, we obtain for every $N \in\mathbb{N}$,
\begin{equation}
  D^2_T\left(\Phi_u^N({\mathbb{Q}}), \Phi_u^N(\widetilde{{\mathbb{Q}}})\right) 
\leq 
\frac{C^NT^N}{N!}D^2_T\left({\mathbb{Q}},\widetilde{{\mathbb{Q}}}\right),
\end{equation}
where $\Phi_u^N$ denotes the $N$-fold composition of $\Phi_u$. Hence 
$\Phi_u^N$ is a contraction for $N$ large enough, thus admitting a unique fixed 
point, which is also the unique fixed point for $\Phi_u$.
Under $\mathbb{P}^u$, the fixed point of $\Phi_u$, 
the coordinate process satisfies
\begin{equation}
    dX_t = \left(\sigma(X_t)\beta\left(t, X_\cdot, \mathbb{P}^u(t), u_t\right) + 
a(X_t)\right)dt + \sigma(X_t)dB^u_t,
\end{equation}
where $B^u$ is a $\mathbb{P}^u$-Brownian motion. 
Following the calculations from Lemma~\ref{eq:lemma_1} that lead to 
\eqref{eq:girsanov_change_}, we get the estimate
\begin{equation}
    (\left\| \mathbb{P}^u\right\|_p)^p = E^u\left[\left|X\right|_T^p\right] \leq 
    \sup_{y\in\bar{\mathcal{D}}}|y|^p,
\end{equation}
where $p\in[1,\infty)$.
\end{proof}

From now on, we will denote the Brownian motion corresponding to $\mathbb{P}^u$ 
by $B^u$. To summarize this section, we have proved the following result 
under Assumption~\ref{assump:beta_meas}-\ref{assump:beta_lip}.
\begin{theorem}
\label{thm:EU-mfeq}
Given $u \in \mathcal{U}$, there exists a unique weak solution to the sticky 
reflected SDE of mean-field type with boundary diffusion
 \begin{equation}
 \label{eq:final_eq_ch_3}
     dX_t = \left(\sigma(X_t)\beta\left(t,X_\cdot, \mathbb{P}^u(t),u_t\right) + 
a(X_t)\right)dt + \sigma(X_t)dB^u_t.
 \end{equation}
 Under $\mathbb{P}^u$ the $t$-marginal distribution of $X_\cdot$ is 
$\mathbb{P}^u(t)$ for $t\in[0,T]$ and $X_\cdot$ is almost surely $C([0,T]; 
\bar{\mathcal{D}})$-valued. Furthermore, $\mathbb{P}^u\in\mathcal{P}_p(\Omega)$.
\end{theorem}
\begin{proof}
 We are left to show that 
 $\mathbb{P}^u\left( X_\cdot\in C([0,T]; \bar{\mathcal{D}}) \right) = 1$,
 all other statements of the theorem have been proved. Let $L^u_T := \mathcal{E}_T(\int_0^\cdot\beta(s, X_\cdot, \mathbb{P}^u(s), u_s)dB_s)$.
 Since $\mathbb{P}(X_\cdot \notin C([0,T]; \bar{\mathcal{D}}) = 0$, 
 \begin{equation}
 \begin{aligned}
    \mathbb{P}^u\left( X_\cdot\notin C([0,T]; \bar{\mathcal{D}}) \right)
    &=
    \mathbb{E}\left[L^u_T 1_{\{X_\cdot\notin C([0,T]; \bar{\mathcal{D}})\}}\right] = 0,
 \end{aligned}
 \end{equation}
 which proves that $X_\cdot$ is $\mathbb{P}^u$-almost surely $C([0,T]; \bar{\mathcal{D}})$-valued.
\end{proof}

 \begin{remark}
 \label{remark:stickyness}
 The drift component $\beta$ is projected in the tangential direction of the 
boundary by $\sigma$ whenever the process is at the boundary (cf. 
\eqref{eq:non_trans_state}). 
 The drift component $a$ is not effected by the transformation. From a modeling 
perspective, the interpretation is that the pedestrian's tangential movement is 
partially controllable but also influenced by other pedestrians through the mean 
field. The normal direction is an uncontrolled delayed reflection.
\end{remark}

\section{Mean-field type optimal control}
\label{sec:control}
Let $E^u$ denote expectation taken under $\mathbb{P}^u$. To apply the stochastic 
maximum principle of \cite{generalSMP}, we make the assumption that the 
mean-field type Girsanov kernel $\beta$ depends linearly on $\mathbb{P}^u$.
\begin{assumption}
\label{assump:law_linear}
Let $\widetilde{\beta} : [0,T]\times\Omega \times \mathbb{R}^d \times U 
\rightarrow \mathbb{R}^d$ 
and let $r_\beta : \mathbb{R}^d \rightarrow \mathbb{R}^d$, and assume that
\begin{equation}
    \beta\left(t, X_\cdot, \mathbb{P}^u(t), u_t\right) = 
\widetilde{\beta}\left(t, X_\cdot, E^u\left[r_\beta(X_t)\right], u_t\right).
  \end{equation}
\end{assumption}
With some abuse of notation, we will continue to denote the Girsanov kernel by 
$\beta$, although from now this refers to $\widetilde{\beta}$. Let $f : [0,T] 
\times \Omega \times \mathbb{R}^d \times U \rightarrow \mathbb{R}$, $g : 
\mathbb{R}^d \times \mathbb{R}^d \rightarrow \mathbb{R}$, $r_f : \mathbb{R}^d 
\rightarrow \mathbb{R}^d$, and  $r_g : \mathbb{R}^d \rightarrow \mathbb{R}^d$.
\begin{assumption}
\label{assump:f_g_meas}
For every $u\in\mathcal{U}$, the process $\left(f(t,X_\cdot, E^u[r_f(X_t)], 
u_t)\right)_t$  is progressively measurable with respect to $\mathbb{F}$ and $(x,y) \mapsto g(x,y)$ is Borel measurable.
\end{assumption}

Consider the finite horizon mean-field type cost functional $J : \mathcal{U} 
\rightarrow \mathbb{R}$,
\begin{equation}
\label{1}
\begin{aligned}
 J(u) := E^u\left[\int_0^T f\left(t, X_\cdot, E^u\left[r_f(X_t)\right], 
u_t\right) dt  + g\left(X_T, E^u\left[r_g(X_T)\right]\right) \right].
\end{aligned}
\end{equation}
The control problem considered in this section is the minimization of $J$ with 
respect to $u\in\mathcal{U}$ under the constraint that the coordinate process 
for any given $u$ satisfies \eqref{eq:final_eq_ch_3}. The integration 
in \eqref{1} is with respect to a measure absolutely continuous with respect to 
$\mathbb{P}$. Changing measure, we get
\begin{equation}
\label{eq:cost_wrt_L}
\begin{aligned}
  J(u) &= E\bigg[ \int_0^T L_t^u f\left(t, X_\cdot, E[L_t^ur_f(X_t)], u_t\right) 
dt \\
&\hspace{3cm} + 
L_T^u g\left(X_T, E[L_T^ur_g(X_T)]\right) \bigg],
\end{aligned}
\end{equation}
where $E$ is the expectation taken under the original probability measure 
$\mathbb{P}$ and  $L^u$ the controlled likelihood process, given by the SDE 
of mean-field type
\begin{equation}
\label{eq:likelihood_dynamics}
    dL_t^u = L_t^u \beta\left(t,X_\cdot, E\left[L_t^ur_\beta(X_t)\right], 
u_t\right)^* dB_t, \quad L_0^u = 1.
\end{equation}

\subsection{Necessary optimality conditions}
After making one final assumption about the regularity of $\beta, f$, and $g$ 
(Assumption~\ref{assump:derivative} below), the stochastic maximum principle 
yields necessary conditions on an optimal control for 
the minimization of \eqref{eq:cost_wrt_L} subject to 
\eqref{eq:likelihood_dynamics}. 
Assumption~\ref{assump:law_linear} and 
\ref{assump:derivative} are stated in their current form for the sake of 
technical, not conceptual, simplicity and may be relaxed.

\begin{assumption}
\label{assump:derivative}
The functions $(t,x,y,u) \mapsto (f,\beta)(t,x,y,u)$ and $(x,y) \mapsto g(x,y)$ 
are twice continuously differentiable with respect to $y$. Moreover, $\beta,f$ 
and $g$ and all their derivatives up to second order with respect to $y$ are 
continuous in $(y,u)$, and bounded.
\end{assumption}
The next result is a slight generalization of \cite[Thm 2.1]{generalSMP}.
The paper \cite{generalSMP} treats 
an optimal control problem of mean-field type with deterministic coefficients.
The approach of \cite{generalSMP}, which goes back to \cite{peng1990general},
extends without any further conditions 
to include random coefficients, as shown in \cite{hosking2012stochastic}.
Moreover, in our case the coefficients are not bounded functions, 
they are linear in the likelihood. This seems to violate the conditions of
\cite[Thm 2.1]{generalSMP} but 
an application of Gr\"{o}nwall's lemma yields $E[(L_t^u)^p] \leq \exp(C(p)t)$ for all
$t\in[0,T]$ and $p\geq 2$, where $C(p)$ is a bounded constant, 
and the estimates of \cite{generalSMP} can be recovered after
an application of H\"{o}lder's inequality.
\begin{theorem}
\label{thm2}
Assume that $(\hat{u}, L^{\hat{u}})$ solves the optimal control problem 
\eqref{eq:cost_wrt_L}-\eqref{eq:likelihood_dynamics}. Then there are two pairs 
of 
$\mathbb{F}$-adapted processes, $(p, q)$ and $(P, 
Q)$, that satisfy the first and second order adjoint equations 
\begin{equation}
\left\{
 \begin{aligned}
  dp_t 
  &= 
  -\Big( q_t\beta^{\hat{u}}_t + 
E\left[q_tL_t^{\hat{u}}\nabla_y\beta^{\hat{u}}_t\right]r_\beta(X_t) 
  \\
  &\hspace{2cm} - 
  f^{\hat{u}}_t - E\left[L^{\hat{u}}_t\nabla_yf^{\hat{u}}_t\right]r_f(X_t)\Big) dt
    + q_tdB_t,
  \\
  p_T 
  &=
  -g^{\hat{u}}_T - E\left[L_T^{\hat{u}}\nabla_yg^{\hat{u}}_T\right]r_g(X_T),
  \\
  \end{aligned}
  \right.
  \end{equation}
  \begin{equation}
  \left\{
  \begin{aligned}
  dP_t 
  &=
  - \Big( \left|\beta^{\hat{u}}_t + 
E\left[L^{\hat{u}}_t\nabla_y\beta^{\hat{u}}_t\right]r_\beta(X_t)\right|^2P_t 
  \\
  &\hspace{2cm}
  + 2Q_t\left(\beta^{\hat{u}}_t + E\left[L^{\hat{u}}_t\nabla_y 
\beta^{\hat{u}}_t\right]r_\beta(X_t)\right)\Big)dt + Q_t dB_t,
  \\
  P_T 
  &= 0,
  \end{aligned}
  \right.
\end{equation}
where $\nabla_y$ denotes differentiation with respect to the 
$\mathbb{R}^d$-valued argument.
Furthermore, $(p,q)$ and $(P, Q)$ satisfy
\begin{equation}
  E\left[\sup_{t\in[0,T]}|p_t|^2 + \int_0^T |q_t|^2dt\right] < \infty,
  \ \
  E\left[ \sup_{t\in[0,T]}|P_t|^2 + \int_0^T |Q_t|^2 dt\right] < \infty,
\end{equation}
and for every $u\in U$ and a.e. $t\in[0,T]$, it holds $\mathbb{P}$-a.s. that 
\begin{equation}
\label{eq:opt_cond}
  \mathcal{H}\left(L^{\hat{u}}_t, u, p_t, q_t\right) - 
\mathcal{H}\left(L^{\hat{u}}_t, \hat{u}_t, p_t, 
q_t\right) 
+
\frac{1}{2}\left[\delta \left(L\beta\right)(t)\right]^TP_t\left[\delta 
\left(L\beta\right)(t)\right] \leq 0,
\end{equation}
where
  $\mathcal{H}(L_t^u, u_t, p_t, q_t) \hspace{-2pt}
  := \hspace{-2pt}
  L_t^u\beta^u_tq_t \hspace{-2pt}
  - \hspace{-2pt}
  L_t^uf^u_t$
  and
\begin{equation}
  \delta (L\beta)(t) 
  := 
  L^{\hat{u}}_t\left(\beta\left(t,X_\cdot, 
  E[L^{\hat{u}}_t r_\beta(X_t)], u\right) - \beta^{\hat{u}}_t\right).
\end{equation}
\end{theorem}
The following local form of the optimality condition \eqref{eq:opt_cond}
can be found in e.g. \cite[pp. 120]{yong1999stochastic}, and will be useful
for computation in Section~\ref{sec:ex}.
If $U$ is a convex set and $\mathcal{H}$ is differentiable with respect to $u$, 
then \eqref{eq:opt_cond} implies 
\begin{equation}
 \label{eq:opt_cond_convex}
 (u-\hat{u}_t)^*\nabla_u\mathcal{H}\left(L^{\hat{u}}_t, \hat{u}_t, p_t, q_t\right)
  \leq 0,\quad \forall\ u\in U,\ \text{a.e. }t\in[0,T],\ \mathbb{P}\mbox{-a.s}.
 \end{equation}

\begin{remark}
Sufficient conditions for weak optimal controls will seldom be satisfied since 
they typically require the Hamiltonian to be convex (or concave) in at least 
state ($L^u_t$) and control ($u_t$). This is false even for the simplest version 
of our problem. Assume that $\beta(t, \omega, y,u) = u$ and $f = 0$, then 
$(\ell,u) \mapsto \mathcal{H}(\ell, u, p, q) = \ell u q$, which is neither 
convex nor concave. However, necessary optimality conditions can be useful as we 
will see in Section~\ref{sec:ex}.
\end{remark}

\subsection{Microscopic interpretation of the mean-field type control problem}
\label{sec:micro}
In this section, we give a microscopic interpretation of the mean-field type control problem \eqref{eq:intro_opt} in the form of an interacting particle 
system (collaboratively) minimizing the social cost.
Our means will be the propagation of chaos result \cite[Thm. 2.6]{lacker2018strong}. We will work under all the assumptions stated so far, but we will use the notation from Section~\ref{sec:model} for $\beta$, $f$, and $g$.

We will fix a closed-loop control and we will assume that all the interacting particles are using this control. This assumption is made in order to extract the approximating property of any solution to the mean-field optimal control problem that is on closed-loop form. In Section~\ref{sec:ex}, we will see examples of such controls.

We introduce an interacting system of sticky reflected SDEs with boundary diffusion. Each equation has an initial value with distribution $\lambda$, where $\lambda$ is a nonatomic measure and $\lambda(\bar{\mathcal{D}}) = 1$. See Remark 10 in \cite{lacker2018strong} for the necessity of the random initial condition.

Consider the measure $\mathbb{P}^{\otimes N}$ on $(\Omega^N,\mathcal{B}(\Omega^N))$, 
the weak solution to a system of $N\in\mathbb{N}$ i.i.d. sticky reflected Brownian motions with boundary diffusion
\begin{equation}
\label{eq:particlesys1}
     dX_t^{N,i} = a(X_t^{N,i})dt + \sigma(X_t^{N,i})dB_t^{i},
     \quad
     X_0^{N,i} = \xi^{N,i}, \quad i=1,\dots, N,
\end{equation}
where $\xi_1,\dots, \xi_N$ are i.i.d. random variables with law $\lambda$ which has support only on $\bar{\mathcal{D}}$, and such that $B^1,\dots,B^N$ are independent $\mathbb{F}$-Wiener processes.
 The functions $a$ and $\sigma$ are defined as in \eqref{eq:abbreviation}. Given controls $u^i\in\mathcal{U}$ (now $\mathbb{F}$-progressively measurable), $i=1,2,\dots$, 
 define the likelihood process $L_{\mathbf{u},t}^{N,i}$ as the solution to
\begin{equation}
    dL^{N,i}_{\mathbf{u},t} 
    =
    L^{N,i}_{\mathbf{u},t}\beta\left(t,X_\cdot^{N,i},\mu^{N}_t, u^i_t\right)^*dB^{i}_t,
    \quad
    L^{N,i}_{\mathbf{u},0}
    = 
    1,\quad i=1,\dots, N,
\end{equation}
where
$\mu^N$ is the empirical measure of the coordinate processes,
\begin{equation*}
     \mu^N := \frac{1}{N}\sum_{i=1}^N\delta_{X^i_\cdot} \in \mathcal{P}(\Omega).
\end{equation*} 
Then $L^N_{\mathbf{u},t} := \prod_{i=1}^N L_{\mathbf{u},t}^{N,i}$ is the Radon-Nikodym derivative for the Girsanov-type change of measure from $\mathbb{P}^{\otimes N}$ to $\mathbb{P}^{N,\mathbf{u}}$, under which the coordinate processes satisfy
\begin{equation}
\label{eq:transformed_partsys}
\left\{
\begin{aligned}
        dX_t^{N,i} &= \left(a(X_t^{N,i})+\sigma(X_t^{N,i})\beta\left(t,X_\cdot^{N,i}, 
        \mu^{N}_t, u^i_t\right)\right)dt + \sigma(X_t^{N,i})d\widetilde{B}_t^{i},
        \\ 
        X_0^{N,i} &= \xi^{N,i}, \quad i=1,\dots, N,
\end{aligned}
\right.
\end{equation}
where $\widetilde{B}^{1},\dots$ are $\mathbb{P}^{N,\mathbf{u}}$-Brownian motions and $\mathbf{u} := (u^1,\dots, u^N)$. 
We note that $\mathbb{P}^{N, \mathbf{u}}$ is the law of a system of interacting diffusion processes.
The social cost of the system \eqref{eq:transformed_partsys} is defined as
\begin{equation}
    \frac{1}{N}\sum_{i=1}^N J^i(\mathbf{u})
    :=
    \frac{1}{N}\sum_{i=1}^N 
    E^{N,\mathbf{u}}\left[
    \int_0^T
    f(t, X^{N,i}_\cdot, \mu^N_t, u^i_t)dt + g(X_T^i, \mu^N_T)
    \right].
\end{equation}
The following theorem is an adaptation of 
\cite[Thm. 2.6]{lacker2018strong} where the drift $b := a + \sigma\beta$ and the Girsanov kernel $\sigma^{-1}b := \beta$.
\begin{theorem}
\label{thm:convergence}
    Let $u\in\mathcal{U}$ be a closed-loop control, i.e.
    $u_t(\omega) = \varphi(\omega_{\cdot\wedge t})$ for some measurable function $\varphi : (\Omega,\mathcal{F}) \rightarrow (U,\mathcal{B}(U))$. Given the control $u$ and a random variable $\xi$ with law $\lambda$ 
    (nonatomic with support only on $\bar{\mathcal{D}}$), 
    the sticky reflected SDE of mean-field type with boundary diffusion
    \begin{equation}
    \label{eq:thm3}
    \left\{
    \begin{aligned}
        dX_t &= \left(a(X_t) + \sigma(X_t)\beta(t, X_\cdot, \mathbb{P}^u(t), \varphi(X_{\cdot\wedge t}))\right)dt + \sigma(X_t)dB_t,
        \\
        X_0 &= \xi,
    \end{aligned}
    \right.
    \end{equation}
    can be approximated by the interacting particle system \eqref{eq:transformed_partsys} with all components using the fixed closed-loop control $u$. Furthermore, the value of the mean-field cost functional $J$ at $u$ is the asymptotic social cost of the interacting particle system as $N\rightarrow\infty$ when all the $X^{N,i}$s are using the fixed control $u$. More specifically,
    \begin{equation}
    \label{eq:poc}
    \lim_{N\rightarrow\infty}D_{T}\left(\mathbb{P}^{N,\mathbf{u}}\circ(X^{N,1}_\cdot,\dots, X^{N,k}_\cdot)^{-1},  (\mathbb{P}^u\circ X_\cdot^{-1})^{\otimes k}\right) = 0,
    \end{equation}
    with $\mathbf{u} = (u,\dots, u)$, and
    \begin{equation}
    \lim_{N\rightarrow\infty}\frac{1}{N}\sum_{i=1}^NJ^i(u,\dots, u)
    \rightarrow J(u).
    \end{equation}
\end{theorem}
\begin{proof}
We denote by $\mathcal{E}(\mathcal{P}(\Omega))$ the smallest $\sigma$-field on $\mathcal{P}(\Omega)$ such that the map $\mu\mapsto \int_\Omega \phi d\mu$ is measurable for all bounded and measurable $\phi : \Omega \rightarrow\mathbb{R}$.
As pointed out in \cite{lacker2018strong}, $\mathcal{E}(\mathcal{P}(\Omega))$ coincides with the Borel $\sigma$-field on $\mathcal{P}(\Omega)$ generated by the topology of weak convergence.

To verify the assumptions of
\cite[Thm. 2.6]{lacker2018strong}, we note that $\beta$ is progressively measurable with respect to $\mathbb{F}$ and that $\beta$ is Lipschitz continuous in the measure-valued argument with respect to $d_{TV}$. This implies condition ($\mathcal{E}$) in \cite{lacker2018strong}, the $\mathcal{E}(\mathcal{P}(\Omega))$-measurability of the function 
\begin{equation}
\begin{aligned}
    F_{s,t} : \mathcal{P}(\Omega) &\rightarrow \mathbb{R},
    \\
    F_{s,t}(\nu) &=\int_{\Omega}\int_s^t\left|\beta(u,\omega,\nu_t)
    - \beta(u,\omega,\mathbb{P}^u(t))\right|^2du\; \nu(d\omega),
\end{aligned}
\end{equation}
the $\tau(\Omega)$-continuity of $F_{s,t}$, and the inequality (2.3) from \cite[Thm. 2.6]{lacker2018strong}.
Furthermore, $\beta$ is bounded, implying condition (A) in \cite{lacker2018strong}.
So the propagation of chaos \eqref{eq:poc} holds.

By \cite[Prop. 2.2]{sznitman1991topics}, the propagation of chaos implies that $\mathcal{P}(\mathcal{P}(\Omega)) \ni M^N := \mathbb{P}^{N,\mathbf{u}}\circ(\mu^N)^{-1} \rightarrow \delta_{\mathbb{P}^u\circ X_\cdot^{-1}}$ in the weak topology. By assumption, $f$ and $g$ are bounded and continuous in the $y$-argument. Hence,
\begin{equation}
\begin{aligned}
    &\lim_{N\rightarrow\infty}\frac{1}{N}\sum_{i=1}^N J^i(u,\dots, u)
    \\
    &=
    \lim_{N\rightarrow\infty}
    \frac{1}{N}\sum_{i=1}^N E^{N,\mathbf{u}}\left[\int_0^T 
    f
    \left(t, X^{N,i}, \mu^N_t, \varphi(X^{N,i}_{\cdot\wedge t})\right)dt + g(X_T^{N,i}, \mu^N_T)
    \right]
     \\
     &=
     \lim_{N\rightarrow\infty}
     E^{N,\mathbf{u}}
     \Bigg[
     \int_0^T \int_{\Omega}
     f\left(t, \omega', \mu^N_t,\varphi(\omega'_{\cdot\wedge t})\right)\mu^N(d\omega')dt 
     \\
     &\hspace{4cm}+ 
     \int_{\Omega}g(\omega'(T), \mu^N_T)\mu^N(d\omega')
     \Bigg]
     \\
     &=
     \lim_{N\rightarrow\infty}
     \int_0^T \int_{\mathcal{P}(\Omega)} \left\{
     \int_{\Omega}
     f\left(t, \omega', \int_{\Omega}r_f(\omega''(t))m(d\omega''),\varphi(\omega'_{\cdot\wedge t})\right)m(d\omega')
     \right\}
     \\
     &\hspace{8cm}M^N(dm)dt 
     \\
     &+ \lim_{N\rightarrow\infty}
     \int_{\mathcal{P}(\Omega)}\int_{\Omega}
     g\left(\omega'(T), \int_{\Omega}r_g(\omega''(T))m(d\omega'')
     \right)m(d\omega')M^N(dm).
    \\
    &=
    E^u\left[\int_0^T f\left(t, X_\cdot, \mathbb{P}^u(t)\right)dt + g\left(X_T, \mathbb{P}^u(T)\right) \right] = J(u).
\end{aligned}
\end{equation}
\end{proof}


\section{Examples}
\label{sec:ex}

As a first step in model validation, experimental results on pedestrian speed 
profiles in a long narrow corridor are replicated in this section. The 
application of the proposed approach also displays the new features it offers 
regarding behavior near walls. From the necessary optimality conditions we 
derive an expression for the optimal control valid in following two toy examples 
and the corridor scenario. The numerical simulations are based on the particle system approximation derived in Section \ref{sec:micro}.

Throughout the rest of this section it is assumed that the compact set $U$ is 
convex and sufficiently large so that all optimal control in the following 
analytical expressions are admissible.
Furthermore, it is assumed that $r_g$ is differentiable and that
$(\hat{u},L^{\hat{u}})$ is optimal for the mean-field type control problem 
\eqref{eq:cost_wrt_L}-\eqref{eq:likelihood_dynamics}. We recall the first order 
adjoint equation,
\begin{equation}
    \label{eq:foaex}
\left\{
\begin{aligned}
    dp_t 
  &= 
  -\Big( q_t\beta^{\hat{u}}_t + 
E\left[q_tL_t^{\hat{u}}\nabla_y\beta^{\hat{u}}_t\right]r_\beta(X_t) 
  \\
  &\hspace{2cm}
  - f^{\hat{u}}_t - 
E\left[L^{\hat{u}}_t\nabla_yf^{\hat{u}}_t\right]r_f(X_t)\Big) dt
    + q_tdB_t,
  \\
  p_T 
  &=
  -g^{\hat{u}}_T - E\left[L_T^{\hat{u}}\nabla_yg^{\hat{u}}_T\right]r_g(X_T).
\end{aligned}
\right.
\end{equation}
Rewriting $E[L_t^{\hat{u}}Y_t]=E^{\hat{u}}[Y_t]$ and changing measure to 
$\mathbb{P}^{\hat{u}}$, \eqref{eq:foaex} becomes
\begin{equation}
 \label{eq:qv2}
\left\{
\begin{aligned}
    dp_t 
  &= 
  - A_t dt
    + q_tdB^{\hat{u}}_t,
  \\
  p_T 
  &=
  -g^{\hat{u}}_T - E^{\hat{u}}\left[\nabla_yg^{\hat{u}}_T\right]r_g(X_T),
\end{aligned}
\right.
\end{equation}
where $A_t :=  E^{\hat{u}}\left[q_t\nabla_y\beta^{\hat{u}}_t\right]r_\beta(X_t) 
  - f^{\hat{u}}_t - E^{\hat{u}}\left[\nabla_yf^{\hat{u}}_t\right]r_f(X_t).$
By the martingale representation theorem (see e.g. \cite[pp. 182]{karatzas1988brownian}) 
$p$ can be written as the conditional expectation
\begin{equation}
\label{eq:pcond}
    p_t = -E^{\hat{u}}\left[ g^{\hat{u}}_T + E^{\hat{u}}[\nabla_y 
g^{\hat{u}}_T]r_g(X_T)\ |\ \mathcal{F}_t\right] + E^{\hat{u}}\left[\int_t^T A_s 
ds\ |\ \mathcal{F}_t\right].
\end{equation}
The theorem applies to our problem since $g$ and its $y$-derivative are assumed to be
bounded. Let
\begin{equation}
    \phi\left(t, X_t\right) := g\left(X_t, 
E^{\hat{u}}[r_g(X_t)]\right) + E^{\hat{u}}[\nabla_y g^{\hat{u}}_t]r_g(X_t).
\end{equation}
By Dynkin's formula,
\begin{equation}
    E^{\hat{u}}[ \phi(T, X_T)\ |\ \mathcal{F}_t] 
    = 
    \phi(t, X_t) 
    + \int_t^T E^{\hat{u}}
    \left[
    \left(\mathcal{G} + \partial_s 
    \right)\phi\left(s,X_s\right)\ |\ \mathcal{F}_t
    \right] ds,
\end{equation}
where $\mathcal{G}$ is the generator of the coordinate process and 
$\partial_{s}$ denotes differentiation with respect to time, working on the two 
remaining arguments of $\phi$. Hence, by applying It$\hat{\text{o}}$'s formula 
on $p$ in \eqref{eq:pcond}, where only $X_\cdot$ contributes to the diffusion 
part, and matching the diffusion parts of that and $p$ from \eqref{eq:qv2}, we 
get
\begin{equation}
    \label{eq:q}
    q_s = -\nabla_x \phi(s,X_s) \sigma(X_s).
\end{equation}
The local optimality condition in the case of
a convex $U$ and coefficients differentiable in $u$, 
given in \eqref{eq:opt_cond_convex} right below Theorem~\ref{thm2},
can be used to write $\hat{u}$ in terms of the other processes. 
To use it, we make the following assumption.
\begin{assumption}
\label{assump:diff_in_u}
 The functions $(t,x,y,u)\mapsto (f,\beta)(t, x, y ,u)$ are differentiable with respect to $u$.
\end{assumption}
\noindent
With Assumption~\ref{assump:diff_in_u} in force, an optimal control $\hat{u}$ satisfies the 
the local optimality condition. The local 
optimality condition is satisfied by any $\hat{u}$ such that
$\nabla_u\mathcal{H}(L_t^{\hat{u}}, \hat{u}_t, p_t, q_t) = 0$
for almost every $t\in[0,T]$, $\mathbb{P}$-a.s., i.e.
\begin{equation}
\label{eq:u}
    q_t\nabla_u \beta^{\hat{u}}_t = \nabla_u f^{\hat{u}}_t, \quad \text{a.e. } t\in[0,T],
   \ \mathbb{P}\text{-a.s.}.
\end{equation}
Since $\mathbb{P}^{\hat{u}}$ is absolutely continuous with respect to 
$\mathbb{P}$, the equality above also holds for almost every $t\in[0,T]$ 
$\mathbb{P}^{\hat{u}}$-a.s. 
We have now at hand an expression for the optimal control whenever we can solve 
\eqref{eq:q}-\eqref{eq:u} for $\hat{u}$.

\subsection{Linear-quadratic problems with convex $U$}

\subsubsection{A non-mean-field example}
Let  $\mathcal{D} \subset \mathbb{R}^d$ be an admissible domain and $\mathbb{P}$ 
the probability measure on the space of continuous paths under which the 
coordinate process solves \eqref{eq:non_trans_state}. Consider the following linear-quadratic 
optimal control problem on $\mathcal{D}$,
 \begin{equation*}
 \left\{
  \begin{aligned}
   \min_{u\in \mathcal{U}}&\ \frac{1}{2}E\left[\int_0^T L_t^u |u_t|^2dt + 
L_T^u |X_T-x_T|^2\right],
   \\
   \text{s.t.}&\ dL_t^u = L_t^u u_t^* dB_t,\quad L_0^u = 1,
  \end{aligned}
  \right.
 \end{equation*}
 where $B$ is a $\mathbb{P}$-Brownian motion. The necessary optimality condition 
\eqref{eq:u} yields
\begin{equation}
    \label{eq:ex1u}
  \hat{u}_t = q^*_t, \quad \mathbb{P}\text{-a.s.},\ \text{a.e. } t\in[0,T].
 \end{equation}
Matching the diffusion coefficients gives us the optimal control,
\begin{equation}
\label{eq:ex1size}
    \hat{u}_t = -\sigma(X_t)\left(X_t - x_T\right), \quad 
\mathbb{P}\text{-a.s.},\ \text{a.e. } t\in[0,T].
\end{equation}
The corresponding likelihood process solves
\begin{equation*}
 dL^{\hat{u}}_t =  -L^{\hat{u}}_t\left(X_t-x_T\right)^*\sigma(X_t)dB_t,\quad 
L^{\hat{u}}_0 = 1,
\end{equation*}
and under $\mathbb{P}^{\hat{u}}$, the optimally controlled path  distribution, 
the coordinate process solves
\begin{equation}
\label{eq:LQ-sol}
 \begin{aligned}
  dX_t 
  &= a(X_t)dt + \sigma(X_t)dB_t
  \\
  &=
  a(X_t)dt + \sigma(X_t)\left(-\sigma(X_t)\left(X_t-x_T\right)dt +
dB^{\hat{u}}_t\right)
  \\
  &=
  \left(a(X_t) - \sigma(X_t)\left(X_t-x_T\right)\right)dt + 
\sigma(X_t)dB^{\hat{u}}_t.
 \end{aligned}
\end{equation}
We have used the fact that $\pi^2 = \pi = \pi^*$, which holds since $\pi$ is an 
orthogonal projection. 

\subsubsection{A mean-field example}

Consider now on some admissible domain $\mathcal{D} \subset \mathbb{R}^d$ the 
mean-field type optimal control problem
 \begin{equation*}
 \left\{
  \begin{aligned}
   \min_{u\in \mathcal{U}}&\ \frac{1}{2}E\left[\int_0^T L_t^u |u_t|^2dt + 
L_T^u \left|X_T-E\left[L_T^u X_T\right]\right|^2\right],
   \\
   \text{s.t.}&\ dL_t^u = L_t^u u_t^* dB_t,\quad L_0^u = 1.
  \end{aligned}
  \right.
 \end{equation*}
 As before, $B$ is a $\mathbb{P}$-Brownian motion, where $\mathbb{P}$ is a 
probability measure on the path space under which the coordinate process solves 
\eqref{eq:non_trans_state}.
Then $E^{\hat{u}}[\nabla_y g^{\hat{u}}_t] = 0$, so (since $r_g(x) = x$ here)
\begin{equation*}
    \nabla_x \phi\left(t, X_t\right) = \left(X_t - 
    E^{\hat{u}}[X_t]\right)^*,
\end{equation*}
and \eqref{eq:u} yields $\hat{u}_t = -\sigma(X_t)(X_t - E^{\hat{u}}[X_t])$ 
$\mathbb{P}$-a.s. for almost every $t\in[0,T]$. Under $\mathbb{P}^{\hat{u}}$ the 
coordinate process solves
\begin{equation*}
  dX_t 
  =
  \left(a(X_t) - 
\sigma(X_t)\left(X_t-E^{\hat{u}}\left[X_t\right]\right)\right)dt + 
\sigma(X_t)dB^{\hat{u}}_t.
\end{equation*}

\subsection{Unidirectional pedestrian motion in a corridor}

Experimental studies have been conducted on the impact of proximity to walls on 
pedestrian speed. Pedestrian speed profiles heavily depend on circumstances like 
location, weather, and congestion. In this section, we will replicate two 
scenarios of unidirectional motion in a confined domain with the proposed 
mean-field type optimal control model. 
Especially, we are interested in how the proposed model behaves on the boundary 
and if boundary movement characteristics can be influenced through the running 
cost $f$. 
Sticky boundaries and boundary diffusion grants our pedestrians controlled 
movement at the boundary. By altering the internal parameters of these effect, 
we are able to shape the mean speed profile at the boundary.

Zanlungo \textit{et al.} \cite{zanlungo2012microscopic} observe that in a 
tunnel connecting a shopping center with a railway station in Osaka, Japan, 
pedestrians tend to lower their walking speed when walking close to the walls. 
The authors obtain a concave cross-section average speed profile from their 
experiment, with its maximum approximately at the center of the corridor. The 
average speed at the center of the corridor is about 10\% higher than that of 
near-wall walkers. 

Daamen and Hoogendoorn \cite{daamen2007flow} on the other hand observe (in a 
controlled environment) pedestrian speeds that are higher at the boundary than 
in the interior of the domain. In their experiment, a unidirectional stream of 
pedestrians walk in a wide corridor that at a certain point, at a 
\textit{bottleneck}, shrinks into a tight corridor. Upstream from the 
bottleneck, pedestrians close to the corridor walls move more freely due to less 
congestion, compared to those at the center of the corridor. The experiment 
results in a cross-section speed profile with more than twice as high average 
pedestrian speed in the low-density regions along corridor walls compared to the 
center of the corridor.

By modeling congestion with simple mean-dependent effects, we can replicate the 
overall shape of the average speed profiles of both 
\cite{zanlungo2012microscopic} and \cite{daamen2007flow} (not the density 
profile, to achieve this one needs a more sophisticated mean-field model). Our 
reason for implementing only mean-dependent effects, not of non-local 
distribution-dependent effects like those considered in for example 
\cite{aurell2018mean}, is solely to simplify the analysis.

Consider a long narrow corridor with walls parallel to the $x$-axis at $y=-0.1$ 
and $y=0.1$. Our analysis requires $\mathcal{D}$ to be $C^2$-smooth, so the 
effective corridor (the corridor perceived by the pedestrians) has rounded 
corners. However, the corners will not have any substantial effect on the 
simulation results since the crowd is initiated so far away from the target that 
under the chosen coefficient values, the pedestrians will not reach it ahead of 
the time horizon $T=1$.
On this domain, crowd behavior is modeled with the following optimal control 
problem
\begin{equation}
\left\{
    \begin{aligned}
     \min_{u_\cdot\in\mathcal{U}}\ &\frac{1}{2}E\left[\int_0^1 L_t^u 
f\left(t,X_\cdot, E\left[L_t^ur_f(X_t)\right], u_t\right) dt + 
L_T^u\left|X_T-x_T\right|^2\right],
     \\
     \text{s.t.}\ & dL_t^u = L_t^u u_t dB_t,\quad L_0^u = 1,
    \end{aligned}
\right.
\end{equation}
where $B$ is a Brownian motion under $\mathbb{P}$, the probability measure under 
which $X_\cdot$ solves \eqref{eq:non_trans_state} with $\gamma = 0.5$, and $x_T$ is the location of 
an exit at the end of the corridor. The choice of $\gamma$ is made so that the plots below are 
visually comparable. The running cost $f$ is of congestion-type,
\begin{equation*}
    f\left(t, X_\cdot, E\left[L_t^u r_f(X_t)\right], u_t\right)
    = \mathcal{C}(X_t)\Big(c_f + h\left(t, X_\cdot, 
E^u\left[r_f(X_t)\right]\right)\Big)u^2_t,
\end{equation*}
where $c_fu^2$, $c_f>0$, is the cost of moving in free space, and $hu^2$ the 
additional cost to move in congested areas. The coefficient $\mathcal{C}(X_t) := 
c_\Gamma1_\Gamma(X_t) + 1_\mathcal{D}(X_t)$, $c_\Gamma>0$, is used to monitor 
$f$ (though it is not our control process) on the boundary $\Gamma$. The cost of 
moving on the boundary is increasing with $c_\Gamma$, so for high $c_\Gamma$ we 
expect  lower speed on the boundary. We know from \eqref{eq:u}-\eqref{eq:q} that
\begin{equation}
\label{eq:feasible}
        q^*_t 
        = 
        \mathcal{C}(X_t)\Big(c_f + h\left(t,X_\cdot, 
E^{\hat{u}}\left[r_f(X_t)\right]\right)\Big)\hat{u}_t,\ \
        q_t
        =
        -(X_t-x_T)^*\sigma(X_t).
\end{equation}
Matching the expressions in \eqref{eq:feasible} yields the optimal control
\begin{equation*}
    \hat{u}_t = \frac{\sigma(X_t)\left(X_t - 
x_T\right)}{\mathcal{C}(X_t)\Big(c_f+h\left(t,X_\cdot, 
E^{\hat{u}}\left[r_f(X_t)\right]\right)\Big)}.
\end{equation*}
It implements the following strategy: move towards the target location $x_T$, 
but scale the speed according to the local congestion. Consider the two 
congestion penalties
\begin{equation}
\label{eq:h1h2}
    h_1 := \left|X_2(t) - E^{\hat{u}}\left[X_2(t)\right]\right|, \hspace{6pt} 
h_2 := \frac{1}{\left|X_2(t) - E^{\hat{u}}\left[X_2(t)\right]\right|},
\end{equation}
where $X_2(t)$ is the second (the $y$-)component of the coordinate process, i.e. 
the component in the direction perpendicular to the corridor walls. Stickiness 
is set to $\gamma=0.5$. The choice of $h$ in \eqref{eq:h1h2} means that we have 
set $r_f(X_t) = X_2(t)$.

The corridor is split into $9$ segments parallel with the corridor walls. The 
mean speed is estimated in each segment for four different values of $c_\Gamma$ 
and the results corresponding to congestion penalty $h_1$ and $h_2$ are 
presented in Figure~\ref{fig:result1} and \ref{fig:result2}, respectively. The 
profiles plotted in Figure~\ref{fig:result1} attains the concave shape observed 
by \cite{zanlungo2012microscopic}, mimicking the fast track in the middle of the 
lane. In Figure~\ref{fig:result2} the profiles follow the convex shape observed 
by \cite{daamen2007flow}, taking into account that movement in the crowded 
center (mean of the group) is costly.   When $c_\Gamma$ is small, the 
pedestrians can travel further on the boundary for the same cost. Heuristically, 
the higher $\gamma$ is the longer it takes for the pedestrian to re-enter 
$\mathcal{D}$ and therefore a high $\gamma$ combined with a small $c_\Gamma$ 
yields the highest boundary speed. 
This effect is evident in the figures, where smaller values of $c_\Gamma$ 
results in higher mean speed at the boundary. We note that we are able to shape 
the mean speed at the boundary by our choice of model parameters.
\begin{figure}[ht!]
    \centering
    \includegraphics[width=1\textwidth, trim = 0cm 7cm 0cm 7cm]{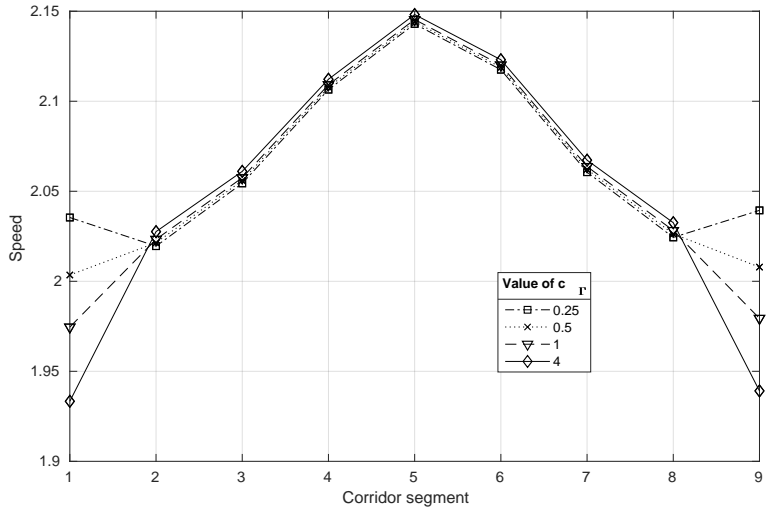}
    \caption{Mean speed in 9 segments of the corridor when $h = h_1$, estimated from $4000$ realizations of the controlled coordinate process.}
    \label{fig:result1}
\end{figure}
~\\ \vspace{1cm}
\begin{figure}[ht!]
    \centering
    \includegraphics[width=1\textwidth, trim = 2cm 7cm 0cm 9cm]{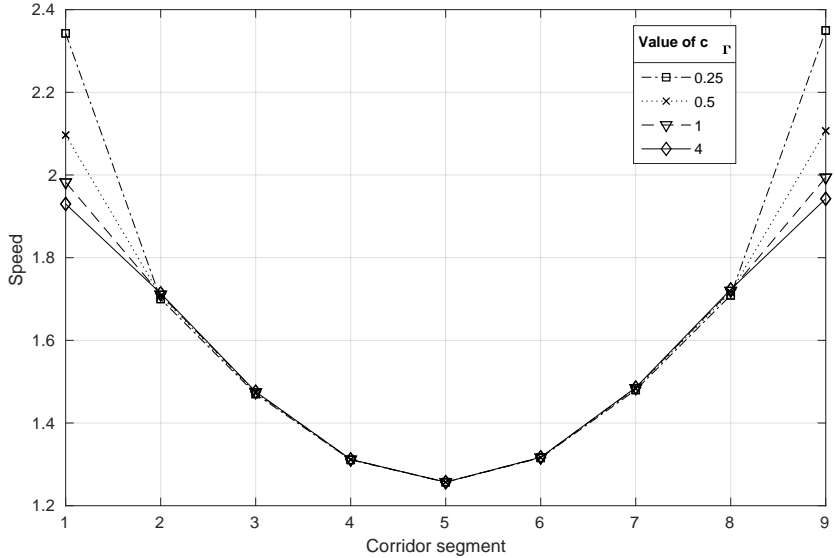}
    \caption{Mean speed in 9 segments of the corridor when $h=h_2$, estimated from $4000$ realizations of the controlled coordinate process.}
    \label{fig:result2}
\end{figure}

\section{Conclusion and discussion}
In this paper, we propose a variation of the 
mean-field approach to crowd modeling based on sticky reflected SDEs which to the best of our knowledge is new. The
proposed model accounts for pedestrians that spend some time at the boundary and that have the possibility to choose a new direction of motion.

We provide conditions for the proposed dynamics to admit a unique weak solution, 
which is the best we can hope for (cf. \cite{engelbert2014stochastic}). Then, we 
consider mean-field type optimal control of the proposed dynamic model and give 
necessary conditions for optimality with a Pontryagin-type stochastic maximum 
principle. There is a microscopic interpretation of the model even on the 
boundary of the domain and thus it has the potential to approximate 
optimal/equilibrium behavior of a pedestrian crowd on a microscopic (individual) 
level. We verify a propagation of chaos result in the uncontrolled case.

Pedestrians do often see and react to walls at a distance. This has been studied 
empirically, experiments are mentioned in the introduction. Force-based models 
can implement repulsing potential forces spiking to infinity at boundaries to 
keep the pedestrians away from the walls and inside the domain, effectively 
making it impossible for any pedestrian to reach a wall. A ranged, 
\textit{nonlocal}, interaction with walls will have a smoothing effect on 
pedestrian density, just like nonlocal pedestrian-to-pedestrian interaction has, 
as is noted in \cite{aurell2018mean}. Nonlocal interaction is an important 
aspect of pedestrian crowd modeling, but cannot give an answer to what will 
happen whenever a pedestrian actually reaches a wall. Interaction with walls at 
a distance can be included in our proposed model either in the drift, as is the 
case in force-based models, or through the cost functional, as in agent-based 
models.

An extension of the proposed framework would be to let the pedestrian control 
its stickiness, i.e. its motion in the normal direction of the boundary at the 
boundary. Stickiness is not necessarily a physical feature of the domain, but 
the time spent on the boundary may be subject to the pedestrian's preference. 
This aspect cannot be described by the proposed model, since the Girsanov change 
of measure does not effect stickiness (cf. Remark~\ref{remark:stickyness}). 
Another extension would be to consider the controlled diffusion case mentioned 
in the introduction.

\bibliographystyle{siam}
\bibliography{references}

\end{document}